\documentclass[a4paper,12pt,reqno]{amsart}
\usepackage{a4wide}
\usepackage{amsmath}
\usepackage{amssymb}
\usepackage{amsthm}
\usepackage{latexsym}
\usepackage{graphicx}
\usepackage[english]{babel}
\newtheorem{prop}[subsection]{Proposition}

\newtheorem{teor}[subsection]{Theorem}
\newtheorem{lema}[subsection]{Lemma}
\newtheorem{cor} [subsection]{Corollary}
\theoremstyle{definition}

\theoremstyle{remark}
\newtheorem{obs} [subsection]{Remark}

\newcommand{\Zng}{$\mathbb Z^n$-graded $S$-module}

\def\sdepth{\operatorname{sdepth}}
\def\qdepth{\operatorname{hdepth}}
\def\hdepth{\operatorname{hdepth}}
\def\depth{\operatorname{depth}}

\def\deg{\operatorname{deg}}

\numberwithin{equation}{section}

\begin{document}

\title[Remarks on the Hilbert depth of powers of the maximal graded ideal]{Remarks on the hdepth of powers of the maximal graded ideal}
%\title{On the quasi depth of monomial ideals}
\author[Silviu B\u al\u anescu, Mircea Cimpoea\c s % and Christian Krattenthaler
       ]{Silviu B\u al\u anescu$^1$ and Mircea Cimpoea\c s$^2$}  
\date{}



\footnotetext[1]{ \emph{Silviu B\u al\u anescu}, University Politehnica of Bucharest, Faculty of
Applied Sciences, %Department of Mathematical Methods and Models, 
Bucharest, 060042, E-mail: silviu.balanescu@stud.fsa.upb.ro}
\footnotetext[2]{ \emph{Mircea Cimpoea\c s}, University Politehnica of Bucharest, Faculty of
Applied Sciences, %Department of Mathematical Methods and Models, 
Bucharest, 060042, Romania and Simion Stoilow Institute of Mathematics, Research unit 5, P.O.Box 1-764,
Bucharest 014700, Romania, E-mail: mircea.cimpoeas@upb.ro,\;mircea.cimpoeas@imar.ro}
% \footnotetext[3]{\emph{Christian Krattenthaler}, Fakult\"at f\"ur Mathematik,
% Universit\"at Wien Oskar-Morgenstern-Platz 1 A-1090 Vienna, Austria, E-mail: Christian.Krattenthaler@univie.ac.at}

\begin{abstract}
Let $\mathbf m=(x_1,\ldots,x_n)$ be the maximal graded ideal of $S:=K[x_1,\ldots,x_n]$. We present a new method
for computing the Hilbert depth of powers of $\mathbf m$.
%We prove that $\qdepth(\mathbf m^t)\leq \left\lceil \frac{n}{t+1} \right\rceil$, for any $t\geq 1$, and,
%moreover, that $\qdepth(\mathbf m^t) = \left\lceil \frac{n}{t+1} \right\rceil$ in the cases (i) $t\leq 2$, (ii) $t\geq n-1$
%and (iii) $n\leq (t+1)(t+3)$. %We strongly believe that the equality holds for any $t\geq 1$.
\end{abstract}

\maketitle

\section{Introduction}

Let $K$ be a field and $S=K[x_1,\ldots,x_n]$ the polynomial ring over $K$.
% We denote $\me=(x_1,\ldots,x_n)$ the maximal graded ideal of $S$.
Let $M$ be a \Zng. A \emph{Stanley decomposition} of $M$ is a direct sum $\mathcal D: M = \bigoplus_{i=1}^rm_i K[Z_i]$ as a 
$\mathbb Z^n$-graded $K$-vector space, where $m_i\in M$ is homogeneous with respect to $\mathbb Z^n$-grading, 
$Z_i\subset\{x_1,\ldots,x_n\}$ such that $m_i K[Z_i] = \{um_i:\; u\in K[Z_i] \}\subset M$ is a free $K[Z_i]$-submodule of $M$. 
We define $\sdepth(\mathcal D)=\min_{i=1,\ldots,r} |Z_i|$ and $$\sdepth(M)=\max\{\sdepth(\mathcal D)|\;\mathcal D\text{ is 
a Stanley decomposition of }M\}.$$ The number $\sdepth(M)$ is called the \emph{Stanley depth} of $M$. 

Herzog, Vladoiu and Zheng show in \cite{hvz} that $\sdepth(M)$ can be computed in a finite number of steps if $M=I/J$, 
where $J\subset I\subset S$ are monomial ideals. %In \cite{rin}, Rinaldo give a computer implementation for this algorithm, 
%in the computer algebra system $\mathtt{CoCoA}$ \cite{cocoa}. 
In \cite{apel}, J.\ Apel restated a conjecture firstly given by Stanley in 
\cite{stan}, namely that $$\sdepth(M)\geq\depth(M),$$ for any \Zng $\;M$. This conjecture proves to be false, in general, for 
$M=S/I$ and $M=J/I$, where $0\neq I\subset J\subset S$ are monomial ideals, see \cite{duval}, but remains open for $M=I$.
For a friendly introduction in the thematic of Stanley depth, we refer to \cite{her}.

Stanley depth is an important combinatorial invariant and deserves a thorough study. The explicit computation of the Stanley depth is a
difficult task, even in some, seemingly, very simple cases as the maximal graded ideal $\mathbf m=(x_1,\ldots,x_n)$ of $S$, see \cite{biro}.
Let $t\geq 1$ be an integer. In \cite[Theorem 2.2]{mir} it was proved that $\sdepth(\mathbf m^t)\leq \left\lceil \frac{n}{t+1} \right\rceil$.
Also, in \cite{mir} it was conjectured that $\sdepth(\mathbf m^t) = \left\lceil \frac{n}{t+1} \right\rceil$ for any $t\geq 1$. This conjecture
holds for $t=1$, see \cite[Theorem 2.2]{biro}, and, also, for $t\geq n-1$, but it is open in general.

Let $M$ be a finitely generated graded $S$-module.
The Hilbert depth of $M$, denoted by $\hdepth(M)$, is the maximal depth of a finitely generated graded $S$-module $N$ with the same Hilbert series as $M$.
Bruns et al. \cite{bruns} proved that $\qdepth(\mathbf m^t)=\left\lceil \frac{n}{t+1} \right\rceil$ for any $t\geq 1$.
In \cite{lucrare2} we proved a new formula for the Hilbert depth of a quotient $J/I$ of two squarefree monomial ideals 
$I\subset J\subset S$, see Theorem \ref{d1}. Also, we extended this method, through polarization, to a quotient $J/I$ of two arbitrary monomial ideals, 
see Proposition \ref{d2}. %We also proved that $\qdepth(J/I)$ is an upper bound for $\sdepth(S/I)$, see Proposition \ref{p1} and Proposition \ref{p2}.
The aim of this note is to study the Hilbert depth of $\mathbf m^t$, where $t\geq 1$ is an integer, from this new perspective.

% In order to do so, we consider the ideal $I_t$ in $R_t:=K[x_1,\ldots,x_{nt}]$ obtained from $\mathbf m^t$ by polarization.
In Theorem \ref{teo} we prove that
$\qdepth(\mathbf m^t)\leq \left\lceil \frac{n}{t+1} \right\rceil$ for any $t\geq 1$. We deduce that 
$\qdepth(\mathbf m)=\sdepth(\mathbf m)=\left\lceil \frac{n}{2} \right\rceil$, see Corollary \ref{cteo}.
In Theorem \ref{main} we prove that $\qdepth(\mathbf m^t)=1$ for $t\geq n-1$ and $\qdepth(\mathbf m^2)=\left\lceil \frac{n}{3} \right\rceil$. 
Also, in Theorem \ref{teo3} we show that $\qdepth(\mathbf m^t) = \left\lceil \frac{n}{t+1} \right\rceil$, if $n\leq (t+1)(t+3)$. 
% We conjecture that $\qdepth(\mathbf m^t) = \left\lceil \frac{n}{t+1} \right\rceil$ for any $t\geq 1$. (Conjecture \ref{pi})

%%%%%%%%%%%%%%%%%%%%%%%%%%%%%%%%%%%%%%%%%%%%%%%%%
\section{Preliminaries}

First, we fix some notations and we recall the main result of \cite{lucrare2}. 

Let $K$ be an infinite field and $S:=K[x_1,\ldots,x_n]$, the ring of polynomials in $n$ variables over $K$.
Let $I\subsetneq J\subset S$ be two square free monomial ideals. We consider the nonnegative integers
$\alpha_k(J/I):=\# \{u\in S\;:\;u\text{ squarefree, with }u\in J\setminus I\text{ and }\deg(u)=k\},\;0\leq k\leq n.$

For all $0\leq d\leq n$ and $0\leq k\leq d$, we consider the integers
\begin{equation}\label{betak}
  \beta_k^d(J/I):=\sum_{j=0}^k (-1)^{k-j} \binom{d-j}{k-j} \alpha_j(J/I).
\end{equation}
Note that, using an inverse formula, from \eqref{betak} we deduce that
\begin{equation}\label{alfak}
  \alpha_k(J/I):=\sum_{j=0}^k \binom{d-j}{k-j} \beta^d_j(J/I).
\end{equation}
With the above notations we have the following result:

\begin{teor}(\cite[Theorem 2.4]{lucrare2})\label{d1}
The Hilbert depth of $J/I$ is
$$\qdepth(J/I):=\max\{d\;:\;\beta_k^d(J/I) \geq 0\text{ for all }0\leq k\leq d\}.$$
\end{teor}

Theorem \ref{d1} can be applied, indirectly, via polarization, in the non squarefree case, as follows. 
If $I\subsetneq J\subset S$ are two monomial ideals, then we consider their polarizations $I^p\subset J^p\subset R$,
where $R$ is a new ring of polynomials obtained from $S$ by adding $N$ new variables:

\begin{prop}\label{d2}
The \emph{Hilbert depth} of $J/I$ is the number
$$\qdepth(J/I):=\qdepth(J^p/I^p)-N.$$
\end{prop}

Also, as the Hilbert depth is an upper bound of the Stanley depth, in particular we have:
\begin{equation}\label{p2}
\sdepth(J/I)\leq \qdepth(J/I).
\end{equation}

\section{Main results}

Let $n\geq 2$ and $S=K[x_1,\ldots,x_n]$.
Let $\mathbf m=(x_1,\ldots,x_n)\subset S$ be the maximal graded ideal and $t\geq 1$ and integer.
We denote by $I_t$, the polarization of the ideal $\mathbf m^t$, in the ring $R_t:=K[x_1,\ldots,x_{nt}]$, that is
$$I_t=\left(\prod_{j=1}^n x_{j}x_{n+j}\cdots x_{n(i_j-1)+j}\;:\;i_1,\ldots,i_j\geq 0\text{ with }i_1+\cdots+i_n=t\right).$$
In other words, the polarization of $x_j^a$ is $x_jx_{n+j}\cdots x_{n(a-1)+j}$, for any $1\leq j\leq n$ and $a\geq 0$.

\begin{lema}\label{l1}
With the above notations, we have
$$\alpha_k(R_t/I_t)=\sum_{j=0}^{t-1} \binom{nt-n-j}{k-j} \binom{n+j-1}{j}\text{ for all }0\leq k\leq nt.$$
\end{lema}

\begin{proof}
Let $u=x_{i_1}x_{i_2}\cdots x_{i_k}$ with $1\leq i_1<i_2<\cdots<i_k\leq nt$ be a (squarefree) monomial.
If $i_1\geq n+1$, i.e. $\{i_1,\ldots,i_k\}\subset \{n+1,\ldots,nt\}$, then $u\notin I_t$. 
Note that there are $\binom{nt-n}{k}$ such monomials. 
On the other hand, if $k\geq t$, $t\leq n$ and $i_t\leq n$, i.e. 
$\{i_1,\ldots,i_t\}\subset \{1,2,\ldots,n\}$, then $u\in I_t$.

Now, assume that $\{i_1,\ldots,i_s\}\subset \{1,\ldots,n\}$ and $\{i_{s+1},\ldots,i_k\}\subset\{n+1,\ldots,nt\}$, where
$1\leq s\leq t-1$. Note that $u\notin I_t$ if and only if there exists some nonnegative integers $a_1,\ldots,a_s$ such
that $0\leq \ell=a_1+\cdots+a_s\leq t-1-s$,
\begin{align*}
& L_{a_1,\ldots,a_s}:=\bigcup_{j=1}^s \{i_j+n,i_j+2n,\ldots,i_j+(a_j-1)n\}\subset \{i_{s+1},\ldots,i_k\}\\
& \text{and }\{i_{s+1},\ldots,i_k\}\setminus L_{a_1,\ldots,a_s} \subset \{n,\ldots,nt\}\setminus 
  \left(L_{a_1,\ldots,a_s}\cup\{i_1+na_1,\ldots,i_s+na_s\}\right).
\end{align*}
It follows that the number of such monomials $u\notin I_t$ is
$$\binom{n}{s}\cdot \sum_{\ell=0}^{t-1-s} \sum_{\substack{a_1,\ldots,a_s\geq 0 \\ a_1+\cdots+a_s=\ell}} \binom{nt-n-s-\ell}{k-s-\ell}
= \binom{n}{s}\cdot \sum_{\ell=0}^{t-1-s} \binom{s+\ell-1}{s-1} \binom{nt-n-s-\ell}{k-s-\ell} .$$
Using the first part of the proof, it follows that
\begin{equation}\label{cico}
\alpha_k(R_t/I_t)=\binom{nt-n}{k}+\sum_{s=1}^{t-1} \binom{n}{s} \sum_{\ell=0}^{t-1-s} \binom{s+\ell-1}{s-1} \binom{nt-n-s-\ell}{k-s-\ell}.
\end{equation}
Denoting $j=s+\ell$ in \eqref{cico}, we get
\begin{align*}
& \alpha_k(R_t/I_t)=\binom{nt-n}{k}+\sum_{s=1}^{t-1} \binom{n}{s} \sum_{j=s}^{t-1} \binom{j-1}{s-1} \binom{nt-n-j}{k-j} = 
\binom{nt-n}{k}+\\
& +\sum_{j=1}^{t-1} \binom{nt-n-j}{k-j} \sum_{s=1}^j \binom{j-1}{s-1}\binom{n}{s} = \binom{nt-n}{k}+
\sum_{j=1}^{t-1} \binom{nt-n-j}{k-j} \binom{n+j-1}{j},
\end{align*}
as required.
\end{proof}

We state the following combinatorial formulas
\begin{equation}\label{magic}
\sum_{j=0}^k (-1)^{k-j}\binom{d-j}{k-j}\binom{n}{j} = \binom{n-d+k-1}{k}\text{ for all }0\leq k\leq d\leq n,
\end{equation}
\begin{equation}\label{magic2}
\sum_{\ell=0}^k (-1)^{k-\ell}\binom{n+\ell-1}{\ell}\binom{d}{k-\ell} = \binom{n-d+k-1}{k}\text{ for all }k,d,n\geq 0,
\end{equation}
which can be easily deduced from the Chu-Vandermonde identity.

\begin{lema}\label{l2}
For any $0\leq k\leq d\leq nt$, we have that
$$\beta_k^d(R_t/I_t)=\sum_{\ell=0}^{t-1} \binom{n+\ell-1}{\ell} \binom{tn-n-d+k-\ell-1}{k-\ell}.$$
\end{lema}

\begin{proof}
From Lemma \ref{l1} and \eqref{betak} it follows that
$$ \beta_k^d(R_t/I_t) = \sum_{j=0}^k (-1)^{k-j} \binom{d-j}{k-j} \sum_{\ell=0}^{t-1} \binom{tn-n-\ell}{j-\ell} \binom{n+\ell-1}{\ell} = $$
\begin{equation}\label{e32}
 = \sum_{\ell=0}^{t-1} \binom{n+\ell-1}{\ell} \sum_{j=\ell}^k (-1)^{k-j}\binom{d-j}{k-j}\binom{tn-n-\ell}{j-\ell}.
\end{equation}
Using the substitution $s=j-\ell$ in \eqref{e32} and applying \eqref{magic}, we obtain
\begin{align*}
& \beta_k^d(R_t/I_t) = \sum_{\ell=0}^{t-1} \binom{n+\ell-1}{\ell} \sum_{s=0}^{k-\ell} (-1)^{k-\ell-s}\binom{d-\ell-s}{k-\ell-s}
\binom{tn-n-\ell}{s} = \\
& = \sum_{\ell=0}^{t-1} \binom{n+\ell-1}{\ell} \binom{tn-n-d+k-\ell-1}{k-\ell},
\end{align*}
as required.
\end{proof}

\begin{obs}\rm
As $\dim(S/\mathbf m^t)=0$, from the definition of the Hilbert depth we have that $\hdepth(S/\mathbf m^t)=0$. 
We mention that this result can be deduce also directly from Lemma \ref{l2}: Indeed, if $d=tn-n$, according to Lemma \ref{l2} we have that
$$\beta_k^{tn-n}(R_t/I_t)=\sum_{\ell=0}^{t-1} \binom{n+\ell-1}{\ell} \binom{k-\ell-1}{k-\ell} = \begin{cases} \binom{n+k-1}{k},& 
0\leq k\leq t-1 \\ 0,& t\leq k\leq nt\end{cases}.$$
On the other hand, we have that
$$\beta_{t}^{tn-n+1}(R_t/I_t) = -\binom{n+t-2}{t-1} < 0.$$
Thus, $\qdepth(R_t/I_t)=nt-n$ and, therefore, $\qdepth(S/\mathbf m^t)=0$.

Note that $\depth(S/\mathbf m^t)=\sdepth(S/\mathbf m^t)=0$, since $\mathbf m$ is an associated prime to $S/\mathbf m^t$.
\end{obs}

\begin{prop}\label{ppp}
For any $0\leq k\leq d\leq nt$, we have that
$$\beta_k^d(I_t)=\binom{nt-d+k-1}{k} - \sum_{\ell=0}^{t-1} \binom{n+\ell-1}{\ell} \binom{tn-n-d+k-\ell-1}{k-\ell}.$$
\end{prop}

\begin{proof}
Since $\alpha_j(I_t)=\binom{nt}{j}-\alpha_j(R_t/I_t)$ for all $0\leq j\leq nt$, from \eqref{magic} we get
$$\beta_k^d(I_t)=\binom{nt-d+k-1}{k} - \beta_k^d(R_t/I_t).$$
Hence, the conclusion follows from Lemma \ref{l2}.
\end{proof}

\begin{prop}\label{propou}
With the above notations, we have that: \large
\begin{enumerate}
\item[(1)] $\beta_{t+1}^{nt-n+\left\lceil \frac{n}{t+1} \right\rceil+1}(I_t)= \binom{n-\left\lceil \frac{n}{t+1} \right\rceil+t-1}{t+1} + 
                    \sum\limits_{\ell=0}^{t-1}(-1)^{t-\ell} \binom{n+\ell-1}{\ell}\binom{\left\lceil \frac{n}{t+1} \right\rceil+1}{t+1-\ell}$.
\item[(2)] $\beta_{k}^{nt-n+\left\lceil \frac{n}{t+1} \right\rceil}(I_t)=\binom{n-\left\lceil \frac{n}{t+1} \right\rceil+k-1}{k}-
\sum\limits_{\ell=0}^{t-1}(-1)^{k-\ell} \binom{n+\ell-1}{\ell}\binom{\left\lceil \frac{n}{t+1} \right\rceil}{k-\ell}$, \normalsize 
for all 

$t+1\leq k\leq nt-n+\left\lceil \frac{n}{t+1} \right\rceil$.
\end{enumerate}
%We have that $\qdepth(\mathbf m^t)\leq \left\lceil \frac{n}{t+1} \right\rceil$.
\end{prop}

\begin{proof}
% From Definition \ref{d2}, it is enough to show that $\qdepth(I_t)\leq nt-n+\left\lceil \frac{n}{t+1} \right\rceil$.
(1) Let $d:=nt-n+\left\lceil \frac{n}{t+1} \right\rceil+1$. From Proposition \ref{ppp} it follows that
\begin{equation}\label{e34}
\beta_{t+1}^d(I_t)= \binom{n-\left\lceil \frac{n}{t+1} \right\rceil+t-1}{t+1} - \sum_{\ell=0}^{t-1}
\binom{n+\ell-1}{\ell}\binom{-\left\lceil \frac{n}{t+1} \right\rceil-1+t-\ell}{t+1-\ell}.
\end{equation}
On the other hand, we have that
\begin{equation}\label{e35}
\binom{-\left\lceil \frac{n}{t+1} \right\rceil-1+t-\ell}{t+1-\ell}=
(-1)^{t-1-\ell}\binom{\left\lceil \frac{n}{t+1} \right\rceil+1}{t+1-\ell}.
\end{equation}
From \eqref{e34} and \eqref{e35} it follows that
$$\beta_{t+1}^d(I_t)= \binom{n-\left\lceil \frac{n}{t+1} \right\rceil+t-1}{t+1} - 
                    \sum_{\ell=0}^{t-1}(-1)^{t+1-\ell}\binom{n+\ell-1}{\ell}
										\binom{\left\lceil \frac{n}{t+1} \right\rceil+1}{t+1-\ell},$$
as required.

% Using the substitution $j=t-\ell$ in \eqref{e36}, we get the required formula.
(2) The proof is similar to the proof of (1).
\end{proof}

\begin{obs}\label{riem}\rm
In order to prove that $\qdepth(I_t)=nt-n+\left\lceil \frac{n}{t+1} \right\rceil$, it suffice to show that
\begin{equation}\label{crookoo}
\beta_{t+1}^{nt-n+\left\lceil \frac{n}{t+1} \right\rceil+1}(I_t)<0\text{ and }
\beta_{k}^{nt-n+\left\lceil \frac{n}{t+1} \right\rceil}(I_t)\geq 0\text{ for all }t+1\leq k\leq nt-n+\left\lceil \frac{n}{t+1} \right\rceil.
\end{equation}
Indeed, from $\beta_{t+1}^{nt-n+\left\lceil \frac{n}{t+1} \right\rceil+1}(I_t)$ it follows that $\qdepth(I_t)\leq nt-n+\left\lceil \frac{n}{t+1} \right\rceil$. Also, since $\beta_k^{nt-n+\left\lceil \frac{n}{t+1} \right\rceil}(I_t)=0$ for $k\leq t-1$ and
$\beta_t^{nt-n+\left\lceil \frac{n}{t+1} \right\rceil}(I_t)=\alpha_t(I_t)>0$, \eqref{crookoo} implies that 
$\qdepth(I_t)\geq nt-n+\left\lceil \frac{n}{t+1} \right\rceil$. 
Also, $\qdepth(I_t)=nt-n+\left\lceil \frac{n}{t+1} \right\rceil$ implies that 
$\qdepth(\mathbf m^t)=\left\lceil \frac{n}{t+1} \right\rceil$, since  $I_t\subset R_t$ is obtained from $\mathbf m^t\subset S$ via polarization and $R_t=S[x_{n+1},x_{n+2},\ldots,x_{nt}]$.
\end{obs}

\begin{teor}\label{teo}
We have that $\qdepth(\mathbf m^t)\leq \left\lceil \frac{n}{t+1} \right\rceil$.
\end{teor}

\begin{proof}
From \eqref{magic2}, it follows that
\begin{equation}\label{cicoo}
\sum_{\ell=0}^{t+1}(-1)^{t+1-\ell}\binom{n-\ell+1}{\ell}\binom{\left\lceil \frac{n}{t+1} \right\rceil+1}{t+1-\ell}=
\binom{n-\left\lceil \frac{n}{t+1} \right\rceil+t}{t+1}.
\end{equation}
From \eqref{cicoo} and Proposition \ref{propou}(1) we henceforth get
\begin{equation}\label{cico2}
\beta_{t+1}^{nt-n+\left\lceil \frac{n}{t+1} \right\rceil+1}(I_t)= -\binom{n-\left\lceil \frac{n}{t+1} \right\rceil+t-1}{t} -
\binom{n+t-1}{t}\left(\left\lceil \frac{n}{t+1} \right\rceil+1 \right) + \binom{n+t}{t+1}.
\end{equation}
On the other hand, we have that
\begin{equation}\label{cico3}
\binom{n+t}{t+1} - \binom{n+t-1}{t}\left(\left\lceil \frac{n}{t+1} \right\rceil+1 \right) = 
\binom{n+t-1}{t}\left(\frac{n-1}{t+1}-\left\lceil \frac{n}{t+1} \right\rceil \right) < 0.
\end{equation}
From \eqref{cico2} and \eqref{cico3} it follows that $$\beta_{t+1}^{nt-n+\left\lceil \frac{n}{t+1} \right\rceil+1}(I_t)<0,$$ and,
therefore, as in Remark \ref{riem}, it follows
that $\qdepth(\mathbf m^t)\leq \left\lceil \frac{n}{t+1} \right\rceil$, as required.
\end{proof}

\begin{cor}\label{cteo}
We have that $\qdepth(\mathbf m)= \left\lceil \frac{n}{2} \right\rceil$.
\end{cor}

\begin{proof}
From Theorem \ref{teo} it follows that $\qdepth(\mathbf m)\leq \left\lceil \frac{n}{2} \right\rceil$. 
On the other hand, from \eqref{p2} and \cite[Theorem 2.2]{biro}, it follows that
$\qdepth(\mathbf m)\geq \sdepth(\mathbf m)=\left\lceil \frac{n}{2} \right\rceil$. Hence, we are done.
\end{proof}

\begin{teor}\label{main}
With the above notations, we have that:
\begin{enumerate}
\item[(1)] $\qdepth(\mathbf m^t)=1$ for $t\geq n-1$.
\item[(2)] $\qdepth(\mathbf m^2)=\left\lceil \frac{n}{3} \right\rceil$.
\end{enumerate}
\end{teor}

\begin{proof}
First, note that, using Theorem \ref{teo}, it suffice to show the $\geq$ inequality. 

(1) Since $t\geq n-1$, that is $\left\lceil \frac{n}{t+1} \right\rceil=1$, from Proposition \ref{propou}(2) it follows that
\begin{equation}\label{cucurigu}
\beta_{k}^{nt-n+1}(I_t)=\binom{n+k-2}{k},\text{ for all }k\geq t+1.
\end{equation}
From \eqref{cucurigu} and the fact that $\beta_k^{nt-n+1}(I_t)=0$ for $k\leq t-1$
and $\beta_t^{nt-n+1}(I_t)=\alpha_t(I_t)>0$, it follows that $\qdepth(I_t)\geq nt-n+1$. Therefore, $\qdepth(I_t)=nt-n$ and, as in
Remark \ref{riem}, this implies $\qdepth(\mathbf m^t)\geq 1$, as required.

(2) Since $t=2$, from Proposition \ref{propou}(2) and the fact that 
$n-\left\lceil \frac{n}{3} \right\rceil=\left\lfloor \frac{2n}{3} \right\rfloor$, it follows that
\begin{equation}\label{e312}
\beta_k^{n+\left\lceil \frac{n}{3} \right\rceil}(I_2)=\binom{\left\lfloor \frac{2n}{3} \right\rfloor+k-1}{k}-
(-1)^k\binom{\left\lceil \frac{n}{3} \right\rceil}{k}+(-1)^{k}n\binom{\left\lceil \frac{n}{3} \right\rceil}{k-1}
\end{equation}
If $k\geq \left\lceil \frac{n}{3} \right\rceil+1$ then, from \eqref{e312}, it follows that
$$\beta_k^{n+\left\lceil \frac{n}{3} \right\rceil}(I_2)=\binom{\left\lfloor \frac{2n}{3} \right\rfloor+k-1}{k} > 0.$$
Also, if $k =\left\lceil \frac{n}{3} \right\rceil$ then, from \eqref{e312} and the fact that $n\geq 2$, it follows that
$$\beta_k^{n+\left\lceil \frac{n}{3} \right\rceil}(I_2)=\binom{n-1}{\left\lceil \frac{n}{3} \right\rceil}+
(-1)^{\left\lceil \frac{n}{3} \right\rceil}\geq 0.$$
Now, assume that $k\leq \left\lceil \frac{n}{3} \right\rceil-1$. From \eqref{e312} we get
\begin{equation}\label{e313}
 \beta_k^{n+\left\lceil \frac{n}{3} \right\rceil}(I_2)=\binom{\left\lfloor \frac{2n}{3} \right\rfloor+k-1}{k}
+ (-1)^k\cdot \frac{nk-\left\lceil \frac{n}{3} \right\rceil+k-1}{\left\lceil \frac{n}{3} \right\rceil - k +1}\binom{\left\lceil \frac{n}{3} \right\rceil}{k}.
\end{equation}
If $k$ is even then, from \eqref{e313} it follows that $\beta_k^{n+\left\lceil \frac{n}{3} \right\rceil}(I_2)>0$, hence, the
only case needed to be considered is $k$ is odd and $3\leq k\leq \left\lceil \frac{n}{3} \right\rceil-1$. If $n\leq 9$ then there is nothing to prove, so we can assume that $n\geq 10$. In order to show that $\beta_k^{n+\left\lceil \frac{n}{3} \right\rceil}(I_2)\geq 0$, by
\eqref{e313}, it suffice to prove that \small
\begin{equation}\label{e314}
\left(\left\lfloor \frac{2n}{3} \right\rfloor+k-1\right)\left(\left\lfloor \frac{2n}{3} \right\rfloor+k-2\right)\cdots 
\left\lfloor \frac{2n}{3} \right\rfloor \geq \left(nk-\left\lceil \frac{n}{3} \right\rceil+k-1\right) \left\lceil \frac{n}{3} \right\rceil 
%\left(\left\lceil \frac{n}{3} \right\rceil-1\right) 
\cdots \left(\left\lceil \frac{n}{3} \right\rceil-k+2\right).
\end{equation}
\normalsize
In order to prove \eqref{e314}, we use induction on $k\geq 3$. If $k=3$, then \eqref{e314} became
\begin{equation}\label{e315}
\left(\left\lfloor \frac{2n}{3} \right\rfloor+2\right)\left(\left\lfloor \frac{2n}{3} \right\rfloor+1\right)\left\lfloor \frac{2n}{3} \right\rfloor \geq  \left(3n+2-\left\lceil \frac{n}{3} \right\rceil\right) \left\lceil \frac{n}{3} \right\rceil 
\left( \left\lceil \frac{n}{3} \right\rceil-1 \right).
\end{equation}
We consider three cases:
\begin{enumerate}
\item[(i)] $n=3p$. Equation \eqref{e315} is equivalent to 
$$(2p+2)(2p+1)2p \geq (8p+2)p(p-1) \Leftrightarrow 8p^3 + 12p^2 + 4p \geq 8p^3 +10p^2 - 2p,$$
which is obviously true.
\item[(ii)] $n=3p+1$. Equation \eqref{e315} is equivalent to 
$$(2p+2)(2p+1)2p \geq (8p+4)(p+1)p \Leftrightarrow 8p^3 + 12p^2 + 4p \geq 8p^3 +12p^2 4p,$$
which is also true.
\item[(iii)] $n=3p+2$. Equation \eqref{e315} is equivalent to
$$ (2p+3)(2p+2)(2p+1) \geq (8p+6)(p+1)p \Leftrightarrow 8p^3 + 24p^2 + 22p + 6 \geq 8p^3 + 14p^2 + 6p, $$
which is again true.
\end{enumerate}
Hence, the initial step of the induction is done. In order to prove the induction step, assume \eqref{e314} holds for $k$.
We have to show that it holds also for $k+2$. In order to do that, it suffice to prove that %\small
\begin{align*}
& \left(\left\lfloor \frac{2n}{3} \right\rfloor+k+1\right)\left(\left\lfloor \frac{2n}{3} \right\rfloor+k\right)
\left(nk-\left\lceil \frac{n}{3} \right\rceil+k-1\right) \geq \\
& \geq \left( \left\lceil \frac{n}{3} \right\rceil-k+1 \right)
\left( \left\lceil \frac{n}{3} \right\rceil-k \right)\left(nk+2n-\left\lceil \frac{n}{3} \right\rceil+k+1\right).
\end{align*}
%\normalsize
This can be proved, by straightforward computations, in a similar manner as \eqref{e315}. 

Now, from all the above considerations, it follows that 
$$\beta_k^{n+\left\lceil \frac{n}{3} \right\rceil}(I_2)\geq 0\text{ for all }0\leq k\leq n+\left\lceil \frac{n}{3} \right\rceil,$$ and, therefore, $\qdepth(I_2)\geq n+\left\lceil \frac{n}{3} \right\rceil$.
Thus, $\qdepth(\mathbf m^2)\geq \left\lceil \frac{n}{3} \right\rceil$, as required.
\end{proof}

%We end our paper with the following result:

\begin{prop}\label{eqi}
The following are equivalent:
\begin{enumerate}
\item[(1)] $\qdepth(\mathbf m^t)=\left\lceil \frac{n}{t+1} \right\rceil$.
\item[(2)] $\sum\limits_{j=0}^{k-t} (-1)^j \binom{n+k-j-1}{k-j}\binom{\left\lceil \frac{n}{t+1} \right\rceil}{j}\geq 0$
for all $t+1\leq k\leq nt-n+\left\lceil \frac{n}{t+1} \right\rceil$.
\item[(3)] $\sum\limits_{j=0}^k (-1)^j \binom{k+t}{j}\binom{n-j}{m-j}\geq 0$ for all $t,k,m,n\geq 1$ such that 
$$m(t+1)+k-1 \leq n \leq (m+1)(t+1)+k-2\text{ and }1\leq k\leq nt-n-t+m.$$
\end{enumerate}
\end{prop}

\begin{proof}
$(1)\Leftrightarrow(2)$. Note that, according to Theorem \ref{teo}, we have that $\qdepth(\mathbf m^t)\leq \left\lceil \frac{n}{t+1} \right\rceil$.

From Proposition \ref{propou}(2) and \eqref{magic2}, using the substitution $j=\ell-t$, it follows that
$$\beta_{k}^{nt-n+\left\lceil \frac{n}{t+1} \right\rceil}(I_t)=
\sum\limits_{\ell=t}^{k}(-1)^{k-\ell} \binom{n+\ell-1}{\ell}\binom{\left\lceil \frac{n}{t+1} \right\rceil}{k-\ell}
=\sum\limits_{j=0}^{k-t} (-1)^j \binom{n+k-j-1}{k-j}\binom{\left\lceil \frac{n}{t+1} \right\rceil}{j},$$
for all $t+1\leq k\leq nt-n+\left\lceil \frac{n}{t+1} \right\rceil$. Hence, the equivalence follows as in Remark \ref{riem}.

$(2)\Rightarrow(3)$. It is clear that $m=\left\lceil \frac{n}{t+1} \right\rceil$, if and only if
\begin{equation}\label{toto1}
m(t+1)\leq n \leq (m+1)(t+1)-1.
\end{equation}
Now, let $n'=n+k-1$, $k'=k-t$. From (2) it follows that
\begin{equation}\label{toto2}
\sum\limits_{j=0}^{k-t} (-1)^j \binom{n+k-j-1}{k-j}\binom{\left\lceil \frac{n}{t+1} \right\rceil}{j}=
\sum\limits_{j=0}^{k'}(-1)^j\binom{n'-j}{k'+t-j}\binom{m}{j}.
\end{equation}
On the other hand, we have that
$$\binom{n'-j}{k'+t-j}\binom{m}{j} = \frac{(n'-j)!m!}{(k'+t-j)!(n'-k-t)!j!(m-j)!} = \frac{(n'-m)!m!}{(k'+t)!(n'-k'-t)!} \times$$
\begin{equation}\label{toto3}
\times \frac{(n'-j)!(k'+t)!}{(n'-m)!(m-j)!(k'+t-j)!j!}
= \frac{\binom{n'}{k'+t}}{\binom{n'}{n-m}}\cdot \binom{k'+t}{j}\binom{n'-j}{n'-m}.
\end{equation}
From \eqref{toto1}, \eqref{toto2} and \eqref{toto3}, be renaming $n'$ with $n$ and $k'$ with $k$, we get the required conclusion.
$(3)\Rightarrow(2)$. The proof is similar.
\end{proof}

For $n,m,k,t\geq 1$, we denote 
\begin{equation}\label{bnmtk}
b(n,m,t,k):=\sum\limits_{j=0}^k (-1)^j \binom{k+t}{j}\binom{n-j}{m-j}.
\end{equation}

\begin{cor}\label{crit}
Let $n,t\geq 1$ and $m=\left\lceil \frac{n}{t+1} \right\rceil$ such that
$$b(n+k-1,m,t,k) \geq 0\text{ for all }1\leq k\leq nt-n-t+m.  $$
Then $\qdepth(\mathbf m^t)=m$.
\end{cor}

\begin{proof}
It follows from Remark \ref{riem} and the proof of Proposition \ref{eqi}.
\end{proof}

\begin{lema}\label{lem1}
We have that $$b(n,m,t,k)=\binom{n-t-k}{m},\text{ for all }1\leq m\leq k.$$
\end{lema}

\begin{proof}
Since $m\leq k$, according to \eqref{magic}, we have that
\begin{align*}
& b(n,m,t,k)=(-1)^m\sum_{j=0}^m (-1)^{m-j} \binom{k+t}{k}\binom{n-j}{m-j} = \\
& = (-1)^m\binom{k+t-n+m-1}{m}= \binom{n-t-k}{m},
\end{align*}
as required.
\end{proof}

Let $n,m,t,k\geq 1$ and $0\leq j\leq k$, such that $m\geq k+1$. We denote
$$f(n,m,t,k,j):=\binom{k+t}{j}\binom{n-j}{m-j}.$$
By straightforward computations, we get:
\begin{equation}\label{e19}
\frac{f(n,m,t,k,j)}{f(n,m,t,k,j+1)}=\frac{(n-k+j+1)(j+1)}{(m-k+j+1)(k+t-j)}.
\end{equation}
From \eqref{e19}, it follows that
\begin{equation}\label{e20}
f(n,m,t,k,j) \geq f(n,m,t,k,j+1)\text{ if and only if }n\geq m+k+t-2j-1+\frac{(m-k)(k+t+1)}{j+1}.
\end{equation}
Since the function $\varphi(j)=m+k+t-2j-1+\frac{(m-k)(k+t+1)}{j+1}$ is decreasing, from \eqref{e20}
it follows that for $n\geq \varphi(0)$ we have that $f(n,m,t,k,j)\geq f(n,m,t,k,j+1)$ for all $0\leq j\leq k-1$.
This allows us to prove the following:

\begin{lema}\label{lem2}
Let $n,m,k,t\geq 1$ such that $n\geq m(t+1)+k-1$. Then:
\begin{enumerate}
\item[(1)] $b(n,m,t,1)\geq 0$.
\item[(2)] If $m\leq k+t$ then $b(n,m,t,k)\geq 0$.
\end{enumerate}
\end{lema}

\begin{proof}
First, note that 
$$\varphi(0)=m+k+t-1+(m-k)(k+t+1)=m(t+1)+k-1+(k-1)(m-t-k).$$
Hence, since $n\geq m(t+1)+k-1$, we have that $n\geq \varphi(0)$ for $k=1$ or $m\geq k+t$.
On the other hand, if $n\geq \varphi(0)$ then, according to a previous remark, we have that
$f(n,m,t,k,j)\geq f(n,m,t,k,j+1)$, for all $0\leq j\leq k-1$, and therefore
$$b(n,m,t,k)=(f(n,m,t,k,0)-f(n,m,t,k,1))+(f(n,m,t,k,2)-f(n,m,t,k,3))+\cdots \geq 0.$$
\end{proof}

\begin{lema}\label{lem3}
Let $n,m\geq 1$, $t\geq 3$ such that $m\geq t+3$ and $m(t+1)+1\leq n\leq (m+1)(t+1)$. Then $b(n,m,t,2)\geq 0$.
\end{lema}

\begin{proof}
We have that 
\begin{equation}\label{curling}
b(n,m,t,2)=\binom{n-2}{m-2}\left(\frac{n(n-1)}{m(m-1)}-(t+2)\frac{n}{m} + \binom{t+2}{2} \right).
\end{equation}
From hypothesis, we have that $\frac{n-1}{m-1} > \frac{n}{m}\geq t+1$ and $\frac{n}{m}\leq t+2$. From \eqref{curling} it follows that
$$ b(n,m,t,2)\geq \binom{n-2}{m-2}\cdot \left( (t+1)^2 - (t+2)^2 + \frac{(t+2)(t+1)}{2} \right) = 
\binom{n-2}{m-2}\cdot \left( \frac{1}{2}t^2-\frac{1}{2}t -2\right). $$
Therefore, $b(n,m,t,2)\geq 0$, since $t\geq 3$.
\end{proof}

Now, we are able to prove the following result:

\begin{teor}\label{teo3}
Let $n,t\geq 1$ such that $n\leq (t+1)(t+3)$. Then $\qdepth(\mathbf m^t)=\left\lceil \frac{n}{t+1} \right\rceil$.
\end{teor}

\begin{proof}
If $t=1$ then the conclusion follows from Corollary \ref{cteo}. Also, if $t=2$ then the conclusion follows from Theorem \ref{main}(2).
Hence, we can assume that $t\geq 3$. Let $m=\left\lceil \frac{n}{t+1} \right\rceil$. Note that, $n\leq (t+1)(t+3)$ implies $m\leq t+3$.
Also, $m(t+1)\leq n\leq m(t+1)+t$.

From Lemma \ref{lem1} it follows that 
\begin{equation}\label{ecu-1}
b(n+k-1,m,t,k) = \binom{n-t-1}{k} \geq 0\text{ for all } m\leq k\leq nt-n-t+m.
\end{equation}
Now, suppose that $k<m$. From Lemma \ref{lem2}(1) we have that 
\begin{equation}\label{ecu-2}
b(n,m,t,1)\geq 0.
\end{equation}
Also, from Lemma \ref{lem3} we have that
\begin{equation}\label{ecu-3}
b(n+1,m,t,2)\geq 0.
\end{equation}
Hence, we can assume that $3\leq k\leq m-1$. Since $m=t+3$ and $k\geq 3$ it follows that $m\leq k+t$. 
Therefore, from Lemma \ref{lem2}(2) it follows that
\begin{equation}\label{ecu-4}
b(n-k+1,m,t,k)\geq 0\text{ for all } 3\leq k\leq m-1.
\end{equation}
The conclusion follows from \eqref{ecu-1}, \eqref{ecu-2}, \eqref{ecu-3}, \eqref{ecu-4} and Corollary \ref{crit}.
\end{proof}

% We end our paper with the following conjecture:

% \begin{conj}\label{pi}
% For any integers $n,t\geq 1$, we have that $\qdepth(\mathbf m^t)=\left\lceil \frac{n}{t+1} \right\rceil$.
% \end{conj}

\subsection*{Acknowledgments}

%We gratefully acknowledge the use of the computer algebra system Cocoa (\cite{cocoa}) for our experiments.
The second author, Mircea Cimpoea\c s, was supported by a grant of the Ministry of Research, Innovation and Digitization, CNCS - UEFISCDI, 
project number PN-III-P1-1.1-TE-2021-1633, within PNCDI III.

% \subsection*{Data availability}

% Data sharing not applicable to this article as no data sets were generated or analyzed during the current study.

% \subsection*{Conflict of interest}

% The authors have no relevant financial or non-financial interests to disclose.

\end{document}